\documentclass[12pt]{article}
\usepackage{amssymb}
\usepackage{amsmath,amsthm}
\usepackage[latin1]{inputenc}
\usepackage[dvips]{graphicx}
\usepackage{hyperref}
\usepackage{color}
\usepackage{cite}
\usepackage{enumerate}
\usepackage{tkz-graph}
\usepackage{tikz}
\usepackage{algorithm, algpseudocode}

\hypersetup{colorlinks=true}

\hypersetup{colorlinks=true, linkcolor=blue, citecolor=blue,urlcolor=blue}

 \newtheorem{remark}{Remark}

 \newtheorem{theorem}[remark]{Theorem}
 \newtheorem{proposition}[remark]{Proposition}
  \newtheorem{observation}[remark]{Observation}
 \newtheorem{corollary}[remark]{Corollary}

\newcommand{\Sd}{\operatorname{Sd}}
\newcommand{\Svarpi}{\operatorname{S\varpi}}

\title{ The Simultaneous Strong Metric Dimension of Graph Families }

\author{A. Estrada-Moreno, C. Garc\'{i}a-G\'{o}mez,
\\Y. Ram\'{i}rez-Cruz, J. A. Rodr\'{i}guez-Vel\'{a}zquez\\{\small Departament d'Enginyeria Inform\`atica i Matem\`atiques,}\\
{\small Universitat Rovira i Virgili,}  {\small Av. Pa\"{\i}sos
Catalans 26, 43007 Tarragona, Spain.} \\{\small
 alejandro.estrada; carlos.garcia; yunior.ramirez; juanalberto.rodriguez\@@urv.cat}}

\begin{document}
\maketitle
\begin{abstract}
Let ${\cal G}$  be a family  of graphs defined on a common (labeled) vertex set $V$. A set $S\subset V$ is said to be a simultaneous strong metric generator for  ${\cal G}$ if it is a strong metric generator for every graph of the family. The minimum cardinality among all simultaneous strong metric generators  for  ${\cal G}$, denoted by  $\Sd_s({\cal G})$, is called   the simultaneous strong metric dimension of ${\cal G}$.
We obtain general results on $\Sd_s({\cal G})$ for  arbitrary families of graphs, with special emphasis on the case of families composed by  a graph and its complement. In particular,  it is shown that the problem of finding the simultaneous strong metric dimension of families of graphs is $NP$-hard,  even when restricted to families of trees. 
\end{abstract}

\section{Introduction}
A {\em generator} of a metric space is a set $S$ of points in the space with the property that every point of the space is uniquely determined by its distances from the elements of $S$. Given a simple and connected graph $G$ with vertex set $V(G)$ and edge set $E(G)$, we consider the metric $d_G:V(G)\times V(G)\rightarrow \mathbb{N}\cup \{0\}$, where $\mathbb{N}$ denotes the set of positive integers and $d_G(x,y)$ denotes the length of a shortest path between $u$ and $v$. The pair $(V(G),d_G)$ is readily seen to be a metric space. A vertex $v\in V(G)$ is said to distinguish two vertices $x$ and $y$ if $d_G(v,x)\ne d_G(v,y)$. 
A set $S\subset V(G)$ is said to be a \emph{metric generator} for $G$ if any pair of vertices of $G$ is
distinguished by some element of $S$. If $S=\{w_1,w_2, \ldots, w_k\}$ is an (ordered) set of vertices, then the {\em metric vector} of a vertex $v \in V(G)$ relative to $S$ is the vector  $(d(v,w_1),d(v,w_2), \ldots, d(v,w_k))$. Thus, $S$ is a metric generator if distinct vertices have distinct  metric vectors relative to $S$.  A minimum cardinality metric generator is called a \emph{metric basis} and
its cardinality, the \emph{metric dimension} of $G$, is denoted by $\dim(G)$. Motivated by the problem of uniquely determining the location of an intruder in a network, by means of a set of devices each of which can detect its distance to the intruder, the concepts of a metric generator and metric basis of a graph were introduced by Slater in \cite{Slater1975} where metric generators were called \emph{locating sets}.  Harary and Melter independently introduced the same concept in \cite{Harary1976}, where metric generators were called \emph{resolving sets}. Applications of the metric dimension to the navigation of robots in networks are discussed in \cite{Khuller1996} and applications to chemistry in \cite{Chartrand2000,Johnson1993,Johnson1998}.  

Seb\"{o} and Tannier  in \cite{Sebo2004} asked the following question for a given  metric generator $T$ of a graph $H$: whenever $H$ is a subgraph of a graph $G$ and the metric vectors of the vertices of $H$ relative to $T$ agree in both $H$ and $G$, is $H$ an isometric subgraph of $G$?  Even though the metric vectors relative to a metric generator of a graph distinguish
all pairs of vertices in the graph, they  do not uniquely determine all distances in a graph as
 was first shown in \cite{Sebo2004}.   It was observed in \cite{Sebo2004} that, if ``metric generator'' is replaced by a stronger notion, namely that of  ``strong  metric generator"   (defined below), then the above question can be answered in the affirmative.

 For $u,v\in V(G)$, the interval $I_G [u, v]$ between $u$ and $v$ is defined as the collection of all vertices that belong to some shortest $u-v$ path. A vertex $w$ {\em strongly resolves} two vertices $u$ and $v$ if $v\in  I_G [u,w]$ or  $u\in I_G [v,w]$ \textit{i.e.}, $d_G(u,w)=d_G(u,v)+d_G(v,w)$ or $d_G(v,w)=d_G(v,u)+d_G(u,w)$.  A set $S$ of vertices in a connected graph $G$ is a \emph{strong metric generator} for $G$ if every two vertices of $G$ are strongly resolved by some vertex of $S$. The smallest cardinality of a strong metric generator for $G$ is called its \emph{strong metric dimension} and is denoted by $\dim_s(G)$.  We say that a strong metric generator for $G$ of cardinality $\dim_s(G)$ is a \emph{strong metric basis} of $G$.

The problem of finding the strong metric dimension of a graph has been studied for several classes of graphs. For instance, this problem was studied for Cayley graphs \cite{Oellermann2007}, distance-hereditary graphs \cite{May2011}, Hamming graphs \cite{Kratica2012}, Cartesian product graphs and direct product graphs \cite{RodriguezVelazquez2014a}, corona product graphs and join graphs \cite{Kuziak2013}, strong product graphs \cite{Kuziak2013c,Kuziak-Erratum}
and convex polytopes \cite{Kratica2012a} . Also, some Nordhaus-Gaddum type results for the strong metric dimension of a graph and its complement are known \cite{Yi2013}. Besides the theoretical results related to the strong metric dimension, a mathematical programming model \cite{Kratica2012a} and metaheuristic approaches \cite{Kratica2008,Mladenovic2012a} for finding this parameter have been developed. For more information the reader is
invited to read the survey 
\cite{Kratica2014} and the references cited therein.

Let ${\mathcal G}=\{G_1,G_2,...,G_k\}$ be a family  of (not necessarily edge-disjoint) connected graphs $G_i=(V,E_i)$ with common vertex set $V$ (the union of whose edge sets is not necessarily the complete graph). Ram\'{i}rez-Cruz, Oellermann  and  Rodr\'{i}guez-Vel\'{a}zquez defined in \cite{Ramirez-Cruz-Rodriguez-Velazquez_2014,Ramirez2014}   a \textit{simultaneous metric generator} for ${\mathcal{G}}$ as a set $S\subset V$ such that $S$ is simultaneously a metric generator for each $G_i$. They introduce the concept of \emph{simultaneous metric basis} of ${\mathcal{G}}$ as a minimum cardinality simultaneous metric generator for ${\mathcal{G}}$, and
its cardinality the \emph{simultaneous metric dimension} of ${\mathcal{G}}$, denoted by $\Sd({\mathcal{G}})$ or explicitly by $\Sd(G_1,G_2,...,G_k )$. Analogously, 
 we define a \emph{simultaneous strong metric generator} for ${\cal G}$ to be a set $S\subset V$ such that $S$ is simultaneously a strong metric generator for each $G_i$. We say that a minimum cardinality  simultaneous strong metric  generator for ${\cal G}$ is a \emph{simultaneous strong metric  basis} of ${\cal G}$, and
its cardinality the \emph{simultaneous strong metric dimension} of ${\cal G}$, denoted by $\Sd_s({\cal G})$ or explicitly by $\Sd_s(G_1,G_2,...,G_t)$.

In this paper we study the problem of finding exact
values or sharp bounds for the simultaneous strong metric dimension. The remainder of the article is organized as follows. In Section \ref{Examples} 
we show that there are some families of graphs for which the simultaneous strong metric dimension can be obtained
relatively easily. To this end we describe the approach developed in \cite{Oellermann2007} of transforming the problem of finding the strong metric dimension of a graph to a vertex cover problem.
In Section \ref{SectionBounds} we obtain  sharp  bounds on the simultaneous strong metric dimension, some of which are generalizations of well known bounds on the strong metric dimension.
In Section \ref{SectionComplement}
we focus  on the particular case of families composed by  a graph and its complement, showing that the problem of finding a simultaneous strong metric generator for $\{G,G^c\}$ can be transformed to the problem of finding a vertex cover of $G$ which, at the same time, is a strong metric generator. Finally, in Section  \ref{SectionComplexity} we show that the problem of finding the simultaneous strong metric dimension of families of trees is $NP$-hard. 

Throughout the paper, we will use the notation $K_n$, $K_{r,s}$, $C_n$ and $P_n$ for complete graphs, complete bipartite graphs, cycle graphs and path graphs of order $n$, respectively. 
For a vertex $v$ of a graph $G$, $N_G(v)$ will denote the set of neighbours or \emph{open neighbourhood} of $v$ in $G$. 
The \emph{closed neighbourhood}, denoted by $N_G[v]$, equals $N_G(v) \cup \{v\}$. If there is no ambiguity, we will simple write $N(v)$ or $N[v]$. Two vertices $x,y\in V(G)$  are \textit{twins} in $G$ if $N_G[x]=N_G[y]$ or $N_G(x)=N_G(y)$. If $N_G[x]=N_G[y]$, they are said to be \emph{true twins}, whereas if $N_G(x)=N_G(y)$ they are said to be \emph{false twins}.  The diameter of a  graph $G$ is denoted by $D(G)$. We recall that a graph $G$ is $2$-antipodal if for each vertex $x\in V(G)$ there exists exactly one vertex $y\in V(G)$ such that $d_G(x,y)=D(G)$. For instance, even cycles $C_{2k}$ and the hypercubes $Q_r$ are $2$-antipodal graphs. Given a graph $G$ and $W\subset V(G)$, we define $\langle W\rangle_G$ as the subgraph of $G$ induced by $W$.
For the remainder of the paper, definitions will be introduced whenever a concept is needed.

\section{Main Tools and Examples}\label{Examples}

It was shown in \cite{Chartrand2000} that $\dim(G)=1$ if and only if $G$ is a path.  It now readily follows that $\dim_s(G)=1$ if and only if $G$ is a path. 
Since any strong metric basis of a path is composed by a leaf, we can state the following remark. 
\begin{remark}\label{familyPaths}
Let $\mathcal{G}$ be a family of
connected graphs defined on a common vertex set. 
Then $\Sd_{s}(\mathcal{G})=1$ if and only if $\mathcal{G}$ is a collection of paths that share a common leaf.

\end{remark}
 
At the other extreme we see that $\dim_s(G)=n-1$ if and only if $G$ is the complete graph  of order $n$.   For a family of graphs we have the following remark.

\begin{remark}
Let $\mathcal{G}$ be a family of
connected graphs defined on a common vertex set. If $K_{n}\in \mathcal{G}$, then 
\[
\Sd_{s}(\mathcal{G})=n-1.
\]
\end{remark}

We now describe the approach developed in \cite{Oellermann2007} of transforming the problem of finding the strong metric dimension of a graph to the vertex cover problem.  A vertex $u$ of $G$ is \emph{maximally distant} from $v$ if for every vertex $w\in N_G(u)$, $d_G(v,w)\le d_G(u,v)$. The collection of all vertices of $G$ that are maximally distant from some vertex of the graph is called the {\em boundary} of the graph, see \cite{Brevsar2008,Caceres2005}, and is denoted by $\partial(G)$\footnote{In fact, the boundary $\partial(G)$ of a graph was defined first in \cite{Chartrand2003d} as the subgraph of $G$ induced by the set mentioned in our article with the same notation.}. If $u$ is maximally distant from $v$ and $v$ is maximally distant from $u$, then we say that $u$ and $v$ are \emph{mutually maximally distant}. Let  $S=\{u\in V(G):$ there exists $v\in V(G)$  such that $u,v$  are mutually maximally distant$\}$.  It is readily seen that $S\subseteq \partial(G)$. 
If $u$ is maximally distant from $v$, and $v$ is not maximally distant from $u$, then $v$ has a neighbour $v_1$, such that $d_G(v_1, u) > d_G(v,u)$, {\em i.e.}, $d_G(v_1, u) =d_G(v,u)+1$. It is easily seen that $u$ is maximally distant from $v_1$. If $v_1$ is not maximally distant from $u$, then $v_1$ has a neighbour $v_2$, such that $d_G(v_2, u) >d_G(v_1,u)$. Continuing in this manner we construct a sequence of vertices $v_1,v_2, \ldots$ such that $d_G(v_{i+1}, u) > d_G(v_i, u)$ for every $i$. Since $G$ is finite this sequence terminates with some $v_k$. Thus for all neighbours $x$ of $v_k$ we have $d_G(v_k,u) \ge d_G(x,u)$, and so $v_k$ is maximally distant from $u$ and $u$ is maximally distant from $v_k$. Hence every boundary vertex belongs to $S$. Certainly  $\partial(G)=S$.

 For some basic graph classes, such as complete graphs $K_n$, complete bipartite graphs $K_{r,s}$,  cycles $C_n$ and hypercube graphs $Q_k$, the boundary is simply the whole vertex set.
It is not difficult to see that this property also holds for all  $2$-antipodal graphs and for all distance-regular graphs.
Notice that the boundary of a tree consists of its leaves. A vertex  of a graph  is a {\em simplicial vertex} if the subgraph induced by  its neighbours is a complete graph. Given a graph $G$, we denote by $\sigma(G)$ the set of simplicial vertices of $G$.
It is readily seen that $\sigma(G)\subseteq \partial(G)$.

We use the notion of ``strong resolving graph'' based on a concept introduced in \cite{Oellermann2007}. The \emph{strong resolving graph} of $G$, denoted by $G_{SR}$,  has vertex set $V(G_{SR}) =V(G)$ where two vertices $u,v$ are adjacent in $G_{SR}$ if and only if $u$ and $v$ are mutually maximally distant in $G$.

A set $S$ of vertices of $G$ is a \emph{vertex cover} of $G$ if every edge of $G$ is incident with at least one vertex of $S$. The \emph{vertex cover number} of $G$, denoted by $\beta(G)$, is the smallest cardinality of a vertex cover of $G$. We refer to a $\beta(G)$-set in a graph $G$ as a vertex cover  of cardinality $\beta(G)$. Oellermann and Peters-Fransen \cite{Oellermann2007} showed that the problem of finding the strong metric dimension of a connected graph $G$ can be transformed to the problem of finding the vertex cover number of $G_{SR}$.

\begin{theorem}{\em \cite{Oellermann2007}}
For any connected graph $G$,
$\dim_s(G) = \beta(G_{SR}).$ \label{lem_oellerman}
\end{theorem}

There are some families of graphs for which the strong resolving graphs can be obtained relatively easily. We state some of these here since we need to refer to these in other sections of the paper.

\begin{observation}\label{observation1}
\mbox{ }
\begin{enumerate}[{\rm(a)}]
\item If $\partial(G)=\sigma(G)$, then $G_{SR}\cong K_{\partial(G)}$. In particular, $(K_n)_{SR}\cong K_n$ and for any tree $T$  with $l(T)$ leaves, $(T)_{SR}\cong K_{l(T)}$.

\item For any $2$-antipodal graph $G$ of order $n$, $G_{SR}\cong \bigcup_{i=1}^{\frac{n}{2}} K_2$. Even cycles are $2$-antipodal. Thus,  $(C_{2k})_{SR}\cong \bigcup_{i=1}^{k} K_2$.

\item For odd cycles $(C_{2k+1})_{SR}\cong C_{2k+1}$.
\end{enumerate}
\end{observation}

From this observation  it is easy to construct several families of graphs ${\cal G}$ satisfying $\Sd_s({\cal G})=\dim_s(G)$, for some $G\in {\cal G}$. 
We introduce the following remarks as straightforward examples. 

\begin{remark}\label{famTressOneResolvesAll}
Let ${\cal G}$ be a family of trees defined on a common vertex set and let $G\in {\cal G}$.   If   $\sigma(G)\supseteq \sigma(G')$, for all $G'\in {\cal G} $, then $\Sd_s({\cal G})=\dim_s(G)$.
\end{remark}

\begin{remark}
Let ${\cal G}$ be a family of $2$-antipodal graphs defined on a common vertex set $V$.  If there exits a partition  $\{V_1,V_2\}$  of $V$ such that for every $u\in V_1$ and every $G\in {\cal G}$, the only vertex diametral to $v$ in $G$ belongs to $V_2$, then 
 $\Sd_s({\cal G})=\dim_s(G)=\frac{|V|}{2}$,  for all $G\in {\cal G} $.
\end{remark}

For a graph $G$ of order $n$ and a graph $H$, the \emph{corona product} of $G$ and $H$, denoted as $G \odot H$, is the graph obtained from $G$ and $H$ by taking one copy of $G$ and $n$ copies of $H$, and joining every vertex $v_i$ of $G$ to every vertex of the $i$-th copy of $H$. The next result is a direct consequence of the fact that no vertex of $G$ is mutually maximally distant with any vertex of $G \odot H$. 

\begin{remark}
Let ${\cal G}=\{G_1,G_2,\ldots,G_k\}$ be a family composed by connected non-trivial graphs, defined on a common vertex set, and let $H$ be a non-trivial graph. Then, for any $i \in \{1,\ldots,k\}$, $$\Sd_s(G_1 \odot H, G_2 \odot H, \ldots, G_k \odot H)=\dim_s(G_i \odot H).$$
\end{remark}

The result above allows to extend results obtained in \cite{Kuziak2013} for $\dim_s(G \odot H)$ to families composed by corona product graphs.

Although it is relatively easy to construct some families of graphs having a given simultaneous strong metric dimension, the problem of computing this parameter is $NP$-hard, even when restricted to families of trees, as we shall show in Section \ref{SectionComplexity}.

\smallskip
\section{Basic Bounds}\label{SectionBounds}
Since every strong metric generator is also a metric generator, for  any family ${\cal G}$ of connected graphs defined  on a common vertex set $V$, 
$$ 1\le \Sd({\cal G})\le \Sd_s({\cal G})\le |V|-1.$$
The case $\Sd({\cal G})=1$ was previously discussed in Remark  
\ref{familyPaths}.  
For the case $\Sd({\cal G})=|V|-1$, consider, for instance, a family $\mathcal{G}$ composed by $r+1$ star graphs of the form $K_{1,r}$, defined on a common vertex set $V$, all of them having different centres. In this case, only one vertex can be excluded from any simultaneous strong metric basis of ${\cal G}$, so that $\Sd_s({\cal G})=|V|-1$. The following result characterizes the graph families for which $\Sd({\cal G})=|V|-1$.

\begin{theorem}\label{ThGeneralUpperBound}
Let $\mathcal{G}$ be a family of connected graphs defined on a common vertex set  $V.$ Then $\Sd_{s}(\mathcal{G})=\left\vert V\right\vert -1$ if and only if for
every pair $u,v\in V,$ there exists a graph $G_{uv}\in \mathcal{G}$ such
that $u$ and $v$ are mutually maximally distant in $G_{uv}.$
\end{theorem}

\begin{proof} 
If $\Sd_{s}(\mathcal{G})=\left\vert V\right\vert -1,$
then for every $v\in V,$ the set $V-\left\{ v\right\} $ is a 
simultaneous strong metric basis of $\mathcal{G}$ and, as a consequence, for every $%
u\in V-\left\{ v\right\} $ there exists a graph $G_{uv}\in \mathcal{G}$ such
that the set $V-\left\{ u,v\right\} $ is not a strong metric generator for $%
G_{uv}$. This means that 
the set $V-\left\{ u,v\right\} $ is not a 
vertex cover of $  \left( G_{uv}\right) _{SR}$ and then $u$ and $v$ must be adjacent in $(G_{uv})_{SR}$ or, equivalently, they are mutually maximally
distant in $G_{uv}.$

Conversely, if for every $u,v\in V$ there exists a graph $G_{uv}\in \mathcal{%
G}$ such that $u$ and $v$ are mutually maximally distant in $G_{uv},$ then
for any strong simultaneous metric basis $B$ of $\mathcal{G}$ either $u\in B$
or $v\in B.$ Hence, all but one element of $V$ must belong to $B.$ Therefore 
$ \left\vert B\right\vert \geq \left\vert V\right\vert -1$ and we can conclude that 
$\Sd_{s}(\mathcal{G})=\left\vert
V\right\vert -1.$
\end{proof}

Given a family $\mathcal{G=}\left\{ G_{1},G_{2},\ldots ,G_{k}\right\} $
of connected graphs defined on a common vertex set $V$, we define $\partial({\cal G})=\displaystyle{\bigcup_{G \in {\cal G}}}\partial(G)$. The following general considerations are
true.

\begin{observation}
For any familly $\mathcal{G=}\left\{ G_{1},G_{2},\ldots ,G_{k}\right\} $ of
connected graphs defined on a common vertex set $V$ and any subfamily $\mathcal{H}\subset \mathcal{G}$.
\[
\Sd_{s}(\mathcal{H})\leq \Sd_{s}(\mathcal{G)}\leq \min \left\{ \left\vert
\partial({\cal G})\right\vert -1,\sum_{i=1}^{k}\dim _{s}(G_{i})\right\} . 
\] 
In particular, 
\[
\max_{i\in \left\{ 1,\ldots ,k\right\} }\left\{ \dim _{s}(G_{i})\right\}
\leq \Sd_{s}(\mathcal{G)}. 
\]
\end{observation}

The above inequalities are sharp. For instance, consider a family $\mathcal{H}_1$ of graphs defined on a vertex set $V$, where some particular  vertex $u\in V$ belongs to a simultaneous strong metric basis $B$. Consider also a family of paths $\mathcal{H}_2$, defined on $V$, sharing all of them this particular vertex $u$ as one of their leaves.
Then $B$ is a simultaneous strong metric basis  of the family $\mathcal{H}_1 \cup \mathcal{H}_2$,  so that $\Sd_s(\mathcal{H}_1 \cup \mathcal{H}_2) =\Sd_s(\mathcal{H}_1)$. On the other hand, a family of trees as the one described in Remark~\ref{famTressOneResolvesAll}, where the set of leaves of one tree contains the sets of leaves of every other tree in the family, satisfies $\Sd_s({\cal G})=|\partial({\cal G})|-1$. Finally, consider the family ${\cal G}=\{G_1,G_2\}$ shown in Figure~\ref{figExampleUpperBoundSum}. It is easy to see that $\Sd_s({\cal G})=\dim_s(G_1)+\dim_s(G_2)=|\partial({\cal G})| - 2 < |\partial({\cal G})| - 1$.

\begin{figure}[h]
\begin{center}

\begin{tikzpicture}
[inner sep=0.7mm, place/.style={circle,draw=black,
fill=white,thick},xx/.style={circle,draw=black,fill=black!99,thick},
transition/.style={rectangle,draw=black,fill=black!20,thick},line width=1pt,scale=0.5]
\coordinate (A1) at (0,4);
\coordinate (B1) at (0,0);
\coordinate (C1) at (2,2);
\coordinate (D1) at (4,2);
\coordinate (E1) at (6,2);
\coordinate (F1) at (8,2);
\coordinate (G1) at (10,4);
\coordinate (H1) at (10,0);

\coordinate (A2) at (14,4);
\coordinate (B2) at (14,0);
\coordinate (C2) at (16,2);
\coordinate (D2) at (18,2);
\coordinate (E2) at (20,2);
\coordinate (F2) at (22,2);
\coordinate (G2) at (24,4);
\coordinate (H2) at (24,0);

\draw[black] (A1) -- (C1) -- (D1) -- (E1) -- (F1) -- (G1);
\draw[black] (B1) -- (C1);
\draw[black] (F1) -- (H1);

\draw[black] (A2) -- (C2) -- (D2) -- (E2) -- (F2) -- (G2);
\draw[black] (B2) -- (C2);
\draw[black] (F2) -- (H2);

\node at (A1) [place]  {};
\coordinate [label=center:{$u_1$}] (u11) at (-0.5,4.5);
\node at (B1) [place]  {};
\coordinate [label=center:{$u_2$}] (u21) at (-0.5,-0.5);
\node at (C1) [place]  {};
\coordinate [label=center:{$v_1$}] (v11) at (2,1.3);
\node at (D1) [place]  {};
\coordinate [label=center:{$v_2$}] (v21) at (4,1.3);
\node at (E1) [place]  {};
\coordinate [label=center:{$v_3$}] (v31) at (6,1.3);
\node at (F1) [place]  {};
\coordinate [label=center:{$v_4$}] (v41) at (8,1.3);
\node at (G1) [place]  {};
\coordinate [label=center:{$u_3$}] (u31) at (10.5,4.5);
\node at (H1) [place]  {};
\coordinate [label=center:{$u_4$}] (u41) at (10.5,-0.5);

\node at (A2) [place]  {};
\coordinate [label=center:{$v_1$}] (v21) at (13.5,4.5);
\node at (B2) [place]  {};
\coordinate [label=center:{$v_2$}] (v22) at (13.5,-0.5);
\node at (C2) [place]  {};
\coordinate [label=center:{$u_1$}] (u12) at (16,1.3);
\node at (D2) [place]  {};
\coordinate [label=center:{$u_2$}] (u22) at (18,1.3);
\node at (E2) [place]  {};
\coordinate [label=center:{$u_3$}] (u32) at (20,1.3);
\node at (F2) [place]  {};
\coordinate [label=center:{$u_4$}] (u42) at (22,1.3);
\node at (G2) [place]  {};
\coordinate [label=center:{$v_3$}] (v32) at (24.5,4.5);
\node at (H2) [place]  {};
\coordinate [label=center:{$v_4$}] (v42) at (24.5,-0.5);

\coordinate [label=center:{$G_1$}] (T1) at (5,-1);
\coordinate [label=center:{$G_2$}] (T2) at (19,-1);

\end{tikzpicture}

\caption{The family ${\cal G}=\{G_1,G_2\}$ satisfies $\Sd_s({\cal G})=\dim_s(G_1)+\dim_s(G_2)=6$.}
\label{figExampleUpperBoundSum}
\end{center}
\end{figure}
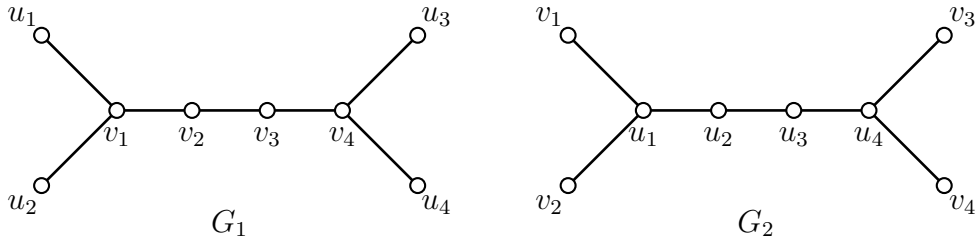

Next, we recall an upper bound for $\dim_s(G)$ obtained in \cite{Kuziak2013}. We say that $X \subseteq V(G)$ is a
\textit{twin-free clique} in $G$  if $X$ is a clique containing no true twins. The \textit{twin-free clique number} of
$G$, denoted by $ \varpi(G)$, is the maximum cardinality among all twin-free cliques in $G$.
\begin{theorem}{\rm \cite{Kuziak2013}}\label{ThTwinsFreeClique}
For any  connected graph $G$    of order $n \ge  2$,  $$\dim_s(G) \le  n - \varpi(G).$$
Moreover, if $D(G) = 2$, then the equality holds.
\end{theorem}  

Our next result is an extension of Theorem \ref{ThTwinsFreeClique} to the case of the simultaneous strong metric dimension. We define a \textit{simultaneous twin-free clique} of a family ${\cal G}$ of graphs as a set which is a twin-free clique in every $G\in {\cal G}$. The \textit{simultaneous twin-free clique number} of
${\cal G}$, denoted by $ \Svarpi({\cal G})$, is the maximum cardinality among all simultaneous twin-free cliques of ${\cal G}$.

\begin{theorem}\label{TheoremTwinclique}
Let ${\cal G}$ be a  family of connected graphs of order $n\ge 2$ defined on a common vertex set. Then
$$\Sd_s({\cal G})\le n-\Svarpi({\cal G}).$$
Moreover, if every graph belonging to ${\cal G}$ has diameter two, then
$$\Sd_s({\cal G}) = n-\Svarpi({\cal G}).$$
\end{theorem}

\begin{proof}
Let $W$ be a simultaneous  twin-free clique in ${\cal G}$ of  maximum cardinality and let $G=(V,E)$ be a graph belonging to ${\cal G}$. We will show that $V - W$ is a strong metric generator for $G$. Since $W$ is a twin-free clique, for any two distinct vertices $u,v\in W$ there exists $s\in V - W$ such that either $s\in N_G(u)$ and $s\notin N_G(v)$ or $s\in N_G(v)$ and $s\notin N_G(u)$. Without loss of generality, we consider $s\in N_G(u)$ and $s\notin N_G(v)$. Thus, $u\in I_G[v,s]$ and, as a consequence,  $s$ strongly resolves $u$ and $v$. Therefore, $\Sd_s({\cal G})\le |V-W| = n - \Svarpi ({\cal G})$.

Now, suppose that every graph $G=(V,E)$ belonging to ${\cal G}$ has diameter two. Let $X\subset V$ be a simultaneous strong metric basis of ${\cal G}$ and let $u,v\in V$, $u\ne v$.   If $d_G(u,v)=2$ or $N_G[u]=N_G[v]$, for some $G\in {\cal G}$, then $u$ and $v$ are mutually maximally distant vertices of $G$, so $u\in X$ or $v\in X$. Hence, for any two distinct vertices $x,y\in V-X$ and any $G\in {\cal G}$ we have $d_G(x,y)=1$ and $N_G[x]\ne N_G[y]$. As a consequence, $V-X$ is a simultaneous twin-free clique of ${\cal G}$ and so $n-\Sd_s({\cal G})=|V-X|\le \Svarpi({\cal G})$. Therefore, $\Sd_s({\cal G})\ge n-\Svarpi({\cal G})$ and the result follows.
\end{proof}

\begin{corollary}
Let ${\cal G}$ be a  family  of graphs of diameter two and order $n\ge 2$ defined on a common vertex set. If ${\cal G}$ contains a triangle-free graph,  then 
$$ n-2\le\Sd_s({\cal G})\le n-1.$$
\end{corollary}

Finally, we recall the following upper bound on $\dim_s(G)$, obtained in \cite{Yi2013}.

\begin{theorem}{\rm \cite{Yi2013}}\label{bound_n_min_diam_G}
For any connected graph $G$ of order $n$, $$\dim_s(G)\leq n-D(G).$$
\end{theorem}

Given a graph family ${\cal G}$  defined on a common vertex set $V$, we define the parameter $\rho({\cal G})=|W|-1$, where $W \subseteq V$ is a maximum cardinality  set such that for every $G\in {\cal G}$ the subgraph  $\langle W \rangle_{G}$ induced by $W$ in $G$ is a path and there exists $w\in W$ which is a common leaf of all these paths.

\begin{theorem}\label{bound_V_min_len_common_path}
Let ${\cal G}$ be a family of graphs  defined on a common vertex set $V$. Then, $$\Sd_s({\cal G}) \le |V|-\rho({\cal G}).$$
\end{theorem}

\begin{proof}
Let $W=\{v_0,v_1, \ldots, v_{\rho({\cal G})}\}\subseteq V$ be a  set for which $\rho({\cal G})$ is obtained. Assume, without loss of generality, that $v_0$ is a common leaf of $\langle W \rangle_G$, for every $G \in {\cal G}$, and let $W'=W-\{v_0\}$. Since no pair of vertices $u,v \in W'$ are mutually maximally distant in any $G \in {\cal G}$,  the set $S=V-W'$ is a simultaneous strong metric generator for ${\cal G}$. Thus, $\Sd_s({\cal G}) \le |S|=|V|-\rho({\cal G})$. 
\end{proof}

The inequality above is sharp. A family of graphs ${\cal G}$ composed by paths having a common leaf is a trivial example where the inequality is reached. In this case, $\rho({\cal G})=|V|-1$, so that $\Sd_s({\cal G})=1=|V|-\rho({\cal G})$. This is not the only circumstance where this occurs. For instance, consider a graph family ${\cal G}$ constructed as follows. Consider   a star graph  $K_{1,r}$  of center $u$ and a complete graph  $K_{r+1}$    defined on a common vertex set $V'$. 
Let $V''$ be a set such that $V'\cap V''=\emptyset$ and let $\{G_1',G_2',\ldots,G_k'\}$ be a family composed by paths defined on $V''$, having a common leaf, say $v$, and let ${\cal G}=\{G_1,H_1,G_2,H_2,\ldots,G_k,H_k\}$ be a graph family
such that every $G_i$ is constructed from $G_i'$ and $K_{1,r}$ by identifying $u$ and $v$, and every $H_i$ is constructed from $G_i'$ and $K_{r+1}$ by identifying $u$ and $v$. For every $w \in V'-\{u\}$, the set $W=V'' \cup \{w\}$ is a maximum cardinality set such that, 
for every graph in $ {\cal G}$, the subgraph induced by  $W$ is a path and there exists $w\in W$ which is a common leaf of all these paths, so that 
 $\rho({\cal G})=|V''|$. Furthermore, the set  $V'-\{u\}$ is a simultaneous strong metric basis of ${\cal G}$ and, as a result, $\Sd_s({\cal G})=r=|V|-\rho({\cal G})$.

In general, the bound shown in Theorem~\ref{bound_V_min_len_common_path} can be efficiently computed, as $\rho({\cal G})$ can be easily computed in $O(|{\cal G}||V|^3$) time using the original Dijkstra's algorithm, which may be accelerated by using special data structures, e.g. Fibonacci heaps \cite {FredmanTarjan1987}.

\smallskip
\section{The Simultaneous Strong Metric Dimension of $\{G,G^c\}$}\label{SectionComplement}

We first consider the following direct consequence of Theorem \ref{ThGeneralUpperBound}.

\begin{corollary}\label{CorollarySD=n-1}
Let $G$ be a graph  of order $n$. Then the following assertions are equivalent.
\begin{itemize}
\item $\Sd_s(G,G^c)=n-1$.
\item $D(G)=D(G^c)=2$.
\end{itemize}
\end{corollary}

\begin{proof}
Let $x,y\in V(G)$. If $D(G)=D(G^c)=2$, then either $x$ and $y$  are diametral in $G$ or they are diametral in $G^c$.  
Hence,  by Theorem \ref{ThGeneralUpperBound} we obtain $\Sd_s(G,G^c)=n-1$.

Now, assume that $D(G)\ge 3$. If $x,u,v,y$ is a shortest path from $x$ to $y$ in $G$, then $x$ and $v$ are not mutually maximally distant in $G$ and, since they are adjacent in $G^c$ and they are not  twins,  they are not mutually maximally distant in $G^c$. Thus, by Theorem \ref{ThGeneralUpperBound} we deduce that  $\Sd_s(G,G^c)\le n-2$.
\end{proof}

The Petersen graph is an example of graphs where $\Sd_s(G,G^c)=n-1$ and the graphs shown in Figure \ref{FigStrongResolCoverNumber} are examples of graphs where $\Sd_s(G,G^c)=n-2$.

From Theorem \ref{ThTwinsFreeClique} and Corollary \ref{CorollarySD=n-1} we derive the next result. 
\begin{theorem}
For any graph $G$ of order $n$ and $D(G)=2$ such that $G^c$ is connected,
$$\Sd_s(G,G^c)\ge n- \varpi(G).$$
Moreover, if $D(G^c)\ge 3$ and $\varpi(G)=2$, then  $$\Sd_s(G,G^c)=n-2.$$
\end{theorem}

Given a graph $G=(V,E)$, we say that a set $S\subset V$ is a \textit{strong resolving cover} for $G$ if $S$ is a vertex cover and a strong metric generator for $G$.

\begin{theorem}\label{ThCover-Generator}
If $G$ is a  connected graph   such that $G^c$ is connected, then any  strong resolving cover of $G$ is a simultaneous strong metric generator for $\{G,G^c\}$.
\end{theorem}
\begin{proof}
Let $W$ be a strong resolving cover of $G$. We shall show that $W$ is a strong metric generator for $G^c$.   We differentiate two cases for any pair $x,y$ of mutually maximally distant vertices in $G^c$.

\vspace{0,2cm}
\noindent Case 1.  $x$ and $y$ are adjacent in $G^c$. 
In this case, $x$ and $y$ are false twins in $G$ (true twins in $G^c$) and so they are mutually maximally distant in $G$. Since $W$ is a strong metric generator for $G$, we conclude that $x\in W$ or $y\in W$.

\vspace{0,2cm}
\noindent Case 2.    $x$ and $y$ are not adjacent in $G^c$. In this case $x$ and $y$ are adjacent in $G$ and, since $W$ is a vertex cover of $G$, we have that $x\in W$ or $y\in W$.

According to the two cases above, $W$ is a vertex cover of $(G^c)_{SR}$ and, as a consequence, $W$ is a strong metric generator for $G^c$.  Therefore, $W$ is a simultaneous strong metric generator for $\{G,G^c\}$.
\end{proof}
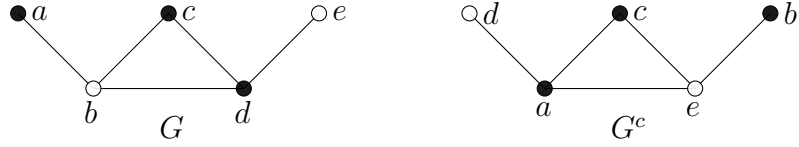
\begin{figure}[h]
\centering
\begin{tikzpicture}[transform shape, inner sep = .7mm]
\draw (-2,1)--(-1,0)--(1,0)--(2,1);
\draw (-1,0)--(0,1)--(1,0);
\draw (4,1)--(5,0)--(7,0)--(8,1);
\draw (5,0)--(6,1)--(7,0);

\filldraw[fill opacity=0.9,fill=black]  (-2,1) circle (0.1cm)  ;
\filldraw[fill opacity=0.9,fill=white]  (-1,0) circle (0.1cm);
\filldraw[fill opacity=0.9,fill=black]  (0,1) circle (0.1cm);
\filldraw[fill opacity=0.9,fill=black]  (1,0) circle (0.1cm);
\filldraw[fill opacity=0.9,fill=white]  (2,1) circle (0.1cm);
\filldraw[fill opacity=0.9,fill=white]  (4,1) circle (0.1cm);
\filldraw[fill opacity=0.9,fill=black]  (5,0) circle (0.1cm);
\filldraw[fill opacity=0.9,fill=black]  (6,1) circle (0.1cm);
\filldraw[fill opacity=0.9,fill=white]  (7,0) circle (0.1cm);
\filldraw[fill opacity=0.9,fill=black]  (8,1) circle (0.1cm);

\coordinate [label=right:{$a$}] (a) at (-1.9,1);
\coordinate [label=right:{$b$}] (b) at (-1.2,-0.31);
\coordinate [label=right:{$c$}] (c) at (0.1,1);
\coordinate [label=right:{$d$}] (d) at(0.8,-0.31);\coordinate [label=right:{$e$}] (e) at (2.1,1);
\coordinate [label=right:{$G$}] (G) at(-0.2,-0.52);
\coordinate [label=right:{$G^c$}] (G) at(5.8,-0.52);

\coordinate [label=right:{$d$}] (d1) at (4.1,1);
\coordinate [label=right:{$a$}] (a1) at (4.8,-0.28);
\coordinate [label=right:{$c$}] (c1) at (6.1,1);
\coordinate [label=right:{$e$}] (e1) at(6.8,-0.28);\coordinate [label=right:{$b$}] (b1) at (8.1,1);
\end{tikzpicture}
\caption{$X_1=\{a,c,d\}$ is a strong resolving cover for $G$ and $X_2=\{a,c,b\}$ is a strong resolving cover for $G^c$. Both $X_1$ and $X_2$ are simultaneous strong metric bases of $\{G,G^c\}$.  }
\label{FigStrongResolCoverNumber}
\end{figure}

 The \textit{strong resolving cover number}, denoted by $\beta_s(G)$, is the minimum cardinality among all the strong resolving covers for $G$. Obviously, for any connected graph of order $n$, 
 \begin{equation} \label{max-cover-strongdim}
 n-1\ge \beta_s(G)\ge \max\{\dim_s(G),\beta(G)\}.
 \end{equation} 

\begin{corollary}
For any   connected graph  $G$ such that $G^c$ is connected,
$$\Sd_s(G,G^c)\le \min\{\beta_s(G),\beta_s(G^c)\}.$$
\end{corollary}

Figure \ref{FigStrongResolCoverNumber} shows a graph $G$ and its complement $G^c$. In this case, $\Sd_s(G,G^c)=\beta_s(G)=\beta_s(G^c)=3>2=\dim_s(G)=\dim_s(G^c)=\beta(G)=\beta(G^c)$.
The graph $G$ shown in Figure \ref{figDimk} satisfies that $\dim_s(G^c)=2<3=\beta_s(G^c)=\Sd_s(G,G^c)=\dim_s(G)<4=\beta_s(G)$. In this case, $\{2,4\}$ is a strong metric basis of $G^c$,  $\{2,3,4\}$ is a $\beta_s(G^c)$-set which is a simultaneous strong metric basis of $\{G, G^c\}$ and, at the same time, it is a strong metric basis of $G$, while $\{2,4,5,6\}$ is a $\beta_s(G)$-set.

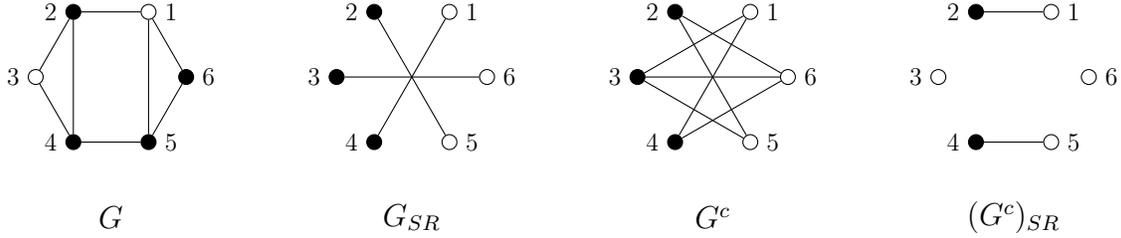
\begin{figure}[h]
\centering
\begin{tikzpicture}[transform shape, inner sep = .7mm]
\def\radius{1} 
\foreach \ind in {1,...,6}
{
\pgfmathparse{360/6*\ind};

\ifthenelse{\ind=2 \OR\ind=4 \OR \ind=5 \OR \ind=6}
{
\node [draw=black, shape=circle, fill=black] (\ind) at (\pgfmathresult:\radius cm) {};
}
{
\node [draw=black, shape=circle, fill=white] (\ind) at (\pgfmathresult:\radius cm) {};
};

\ifthenelse{\ind=2 \OR \ind=3 \OR \ind=4}
{
\node [scale=.8] at ([xshift=-.3 cm]\ind) {$\ind$};
}
{
\node [scale=.8] at ([xshift=.3 cm]\ind) {$\ind$};
};
}

\foreach \ind in {2,5,6}
{
\draw[black] (1) -- (\ind);
}
\foreach \ind in {2,3,5}
{
\draw[black] (4) -- (\ind);
}
\draw[black] (2) -- (3);
\draw[black] (5) -- (6);
\node at ([shift=({\radius/2,-1})]4) {$G$};

\pgfmathsetmacro{\traslation}{(2*\radius+2};
\foreach \ind in {1,...,6}
{
\pgfmathparse{360/6*\ind};
\ifthenelse{\ind=2 \OR \ind=3 \OR \ind=4}
{
\node [draw=black, shape=circle, fill=black, xshift=\traslation cm] (s\ind) at (\pgfmathresult:\radius cm) {};
}
{
\node [draw=black, shape=circle, fill=white, xshift=\traslation cm] (s\ind) at (\pgfmathresult:\radius cm) {};
};

\ifthenelse{\ind=2 \OR \ind=3 \OR \ind=4}
{
\node [scale=.8] at ([xshift=-.3 cm]s\ind) {$\ind$};
}
{
\node [scale=.8] at ([xshift=.3 cm]s\ind) {$\ind$};
};
}

\foreach \ind in {1,...,3}
{
\pgfmathparse{int(\ind+3)};
\draw[black] (s\ind) -- (s\pgfmathresult);
}
\node at ([shift=({\radius/2,-1})]s4) {$G_{SR}$};

\pgfmathsetmacro{\traslation}{(4*\radius+4};
\foreach \ind in {1,...,6}
{
\pgfmathparse{360/6*\ind};
\ifthenelse{\ind=2 \OR \ind=3 \OR \ind=4}
{
\node [draw=black, shape=circle, fill=black, xshift=\traslation cm] (c\ind) at (\pgfmathresult:\radius cm) {};
}
{
\node [draw=black, shape=circle, fill=white, xshift=\traslation cm] (c\ind) at (\pgfmathresult:\radius cm) {};
};
\ifthenelse{\ind=2 \OR \ind=3 \OR \ind=4}
{
\node [scale=.8] at ([xshift=-.3 cm]c\ind) {$\ind$};
}
{
\node [scale=.8] at ([xshift=.3 cm]c\ind) {$\ind$};
};
}

\foreach \ind in {2,...,4}
{
\draw[black] (c6) -- (c\ind);
}
\foreach \ind in {3,4}
{
\draw[black] (c1) -- (c\ind);
}
\foreach \ind in {2,3}
{
\draw (c5) -- (c\ind);
}
\node at ([shift=({\radius/2,-1})]c4) {$G^c$};

\pgfmathsetmacro{\traslation}{(6*\radius+6};
\foreach \ind in {1,...,6}
{
\pgfmathparse{360/6*\ind};
\ifthenelse{\ind=2 \OR \ind=4}
{
\node [draw=black, shape=circle, fill=black, xshift=\traslation cm] (d\ind) at (\pgfmathresult:\radius cm) {};
}
{
\node [draw=black, shape=circle, fill=white, xshift=\traslation cm] (d\ind) at (\pgfmathresult:\radius cm) {};
};
\ifthenelse{\ind=2 \OR \ind=3 \OR \ind=4}
{
\node [scale=.8] at ([xshift=-.3 cm]d\ind) {$\ind$};
}
{
\node [scale=.8] at ([xshift=.3 cm]d\ind) {$\ind$};
};
}

\draw (d1) -- (d2);
\draw (d4) -- (d5);
\node at ([shift=({\radius/2,-1})]d4) {$(G^c)_{SR}$};
\end{tikzpicture}
\caption{The $\beta_s(G^c)$-set $\{2,3,4\}$ is a simultaneous strong metric basis of $\{G,G^c\}$.}
\label{figDimk}
\end{figure}

\begin{theorem}\label{GEqualGcSRtwoDirection}
Let $G$ be a connected graph such that $D(G^c)=2$ and let $S\subset V(G)$. Then  the following assertions are equivalent.
\begin{enumerate}[{\rm (a)}]
\item $S$ is a simultaneous strong metric generator for $\{G,G^c\}$.

\item $S$ is a strong resolving cover for $G$.
\end{enumerate}
\end{theorem}

\begin{proof}
Let $G=(V,E)$. Since $D(G^c)=2$, two vertices $x,y\in V$ are mutually maximally distant in $G^c$ if and only if $d_{G^c}(x,y)=2$ or $N_{G^c}[x]=N_{G^c}[y]$. Hence, $(G^c)_{SR}=(V,E\cup E')$, where $E'=\{\{x,y\}:\; N_G(x)=N_G(y) \}.$

Let $S$ be a simultaneous strong metric generator for $\{G,G^c\}$.
 Since $S$ is a strong metric generator for $G^c$, we deduce that $S$ is a vertex cover of $(G^c)_{SR}=(V,E\cup E')$, and as a consequence, for any edge  $\{x,y\}\in E$, we have that  $x\in S$ or $y\in S$.  Hence, $S$ is a  strong metric generator for $G$ and a vertex cover of $G$. By Theorem \ref{ThCover-Generator} we conclude the proof.
\end{proof}

From Theorem \ref{GEqualGcSRtwoDirection} we deduce the following result.

\begin{corollary}\label{Dim=beta-s}
For any connected graph $G$ such that $D(G^c)=2$,
$$\Sd_s(G,G^c)=\beta_s(G).$$
\end{corollary}

In order to present the next result, we need to introduce some new notation and terminology. 
Given a graph $G$ such that $V(G)\ne \partial(G)$, we
  define the \textit{interior subgraph} of $G$ as the subgraph $\mathring{G}$ induced by  $V(G)-\partial (G)$.
 The parameter $\mathring{\beta}(G)$  is defined as follows.
 
$$\mathring{\beta}(G)=\left\{\begin{array}{ll}
0& \text{if } V(G)= \partial(G)\\
&\\
\beta( \mathring{G} )&\text{otherwise.}
\end{array}
\right.$$
\begin{corollary}\label{CorollaryInequalityMax}
For any     connected graph  $G$ such that $D(G^c)=2$,  $$\Sd_s(G,G^c)\ge \max\{\dim_s(G)+\mathring{\beta}(G),\beta(G)\}.$$
\end{corollary}

\begin{proof}
By Theorem \ref{GEqualGcSRtwoDirection} and Eq.(\ref{max-cover-strongdim}) we have that $\Sd_s(G,G^c)\ge\beta(G)$. It only remains to prove that $\Sd_s(G,G^c)\ge\dim_s(G)+\mathring{\beta}(G)$. If $V(G)= \partial(G)$, then $\mathring{\beta}(G)=0$, and by Theorem \ref{GEqualGcSRtwoDirection} and Eq.(\ref{max-cover-strongdim})  we have  $\Sd_s(G,G^c)\ge\dim_s(G)=\dim_s(G)+\mathring{\beta}(G)$. Assume that $V(G)\ne  \partial(G)$. Let $B$ be a simultaneous strong metric basis of $\{G,G^c\}$, and let $B_1=B\cap \partial(G)$ and $B_2=B-B_1$. Clearly, $|B_1|\ge \dim_s(G)$. Moreover, since no vertex of $B_1$ covers edges of 
$\mathring{G}$, by 
Theorem \ref{GEqualGcSRtwoDirection}  we conclude that $B_2$ is a vertex cover of $\mathring{G}$, so that $|B_2|\ge\beta(\mathring{G})$. Therefore, $\Sd_s(G,G^c)=|B|=|B_1|+|B_2|\ge\dim_s(G)+\mathring{\beta}(G)$. 
\end{proof}

To illustrate this result we take the graph $G$ shown in Figure
\ref{figSdBeta}. In this case  $\Sd_s(G,G^c)=\beta(G)=5>4=\dim_s(G)+\mathring{\beta}(G)$.
In contrast, the equality $\Sd_s(G,G^c)=\dim_s(G)+\mathring{\beta}(G)$ is satisfied for any graph constructed as follows. Let $r,s\ge 2$ and $t\ge 3$ be three integers and let $G$ be the graph constructed from $K_r,K_s$ and $P_t$ by identifying one vertex of $K_r$ with one leaf of $P_t$ and one vertex of $K_s$ with the other leaf of $P_t$. In this case  
$\Sd_s(G,G^c)=r+s+\lfloor \frac{t}{2}\rfloor-1$, $\dim_s(G)=r+s-1$, $\beta(G)=r+s+\lfloor \frac{t}{2}\rfloor-2$
 and $\mathring{\beta}(G)=\beta(\mathring{G})=\lfloor \frac{t}{2}\rfloor.$ Hence,  $\Sd_s(G,G^c)=\dim_s(G)+\mathring{\beta}(G)>\beta(G)$.

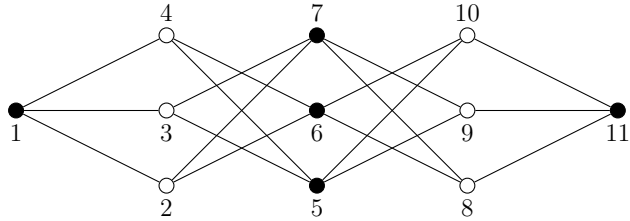
\begin{figure}[!ht]
\centering
\begin{tikzpicture}[transform shape, inner sep = .7mm]
\def\spacey{1};
\def\spacex{2};
\foreach \x/\c in {1/0,3/1,6/2,9/3,11/4}
{
\ifthenelse{\x=3 \OR \x=9}
{
\node [draw=black, shape=circle, fill=white] (\x) at (\c*\spacex cm, 0) {};
}
{
\node [draw=black, shape=circle, fill=black] (\x) at (\c*\spacex cm, 0) {};
};
\node [scale=.8] at ([yshift=-.3 cm]\x) {$\x$};
}
\foreach \x/\c in {4/1,7/2,10/3}
{
\ifthenelse{\x=7}
{
\node [draw=black, shape=circle, fill=black] (\x) at (\c*\spacex cm, \spacey cm) {};
}
{
\node [draw=black, shape=circle, fill=white] (\x) at (\c*\spacex cm, \spacey cm) {};
};
\node [scale=.8] at ([yshift=.3 cm]\x) {$\x$};
}
\foreach \x/\c in {2/1,5/2,8/3}
{
\ifthenelse{\x=5}
{
\node [draw=black, shape=circle, fill=black] (\x) at (\c*\spacex cm, -\spacey cm) {};
}
{
\node [draw=black, shape=circle, fill=white] (\x) at (\c*\spacex cm, -\spacey cm) {};
};
\node [scale=.8] at ([yshift=-.3 cm]\x) {$\x$};
}
\foreach \x in {2,3,4}
{
\draw (1) -- (\x);
}
\foreach \x in {6,7}
{
\draw (2) -- (\x);
}
\foreach \x in {5,7}
{
\draw (3) -- (\x);
}
\foreach \x in {5,6}
{
\draw (4) -- (\x);
}
\foreach \x in {9,10}
{
\draw (5) -- (\x);
}
\foreach \x in {8,10}
{
\draw (6) -- (\x);
}
\foreach \x in {8,9}
{
\draw (7) -- (\x);
}
\foreach \x in {8,9,10}
{
\draw (11) -- (\x);
}
\end{tikzpicture}
\caption{The sets $\{1,5,6,7\}$ and $\{5,6,7,11\}$ are the only strong metric bases of $G$, while  $\{1,5,6,7,11\}$ is the only $\beta(G)$-set which is a strong metric generator of $G$.}
\label{figSdBeta}
\end{figure}

\begin{corollary}\label{CorollaryEqualityto-dim(G)}
Let $G$ be a connected graph  such that $D(G^c)=2$. Then the following assertions hold.

\begin{itemize}
\item $\Sd_s(G,G^c)=\dim_s(G)$ if and only if there exists a strong metric basis of $G$ which is a vertex cover of $G$. 
\item $\Sd_s(G,G^c)=\beta(G)$ if and only if  there exists a $\beta(G)$-set which is a strong metric generator of $G$. 
\end{itemize}
\end{corollary}

\begin{figure}[!ht]
\centering
\begin{tikzpicture}[transform shape, inner sep = .7mm]
\def\spacey{1};
\def\spacex{2};
\foreach \x in {1,2,3}
{
\pgfmathparse{\spacex*(\x-1)};
\ifthenelse{\x=2}
{
\node [draw=black, shape=circle, fill=black] (\x) at (\pgfmathresult cm, 0) {};
}
{
\node [draw=black, shape=circle, fill=white] (\x) at (\pgfmathresult cm, 0) {};
};
\node [scale=.8] at ([yshift=-.3 cm]\x) {$\x$};
}
\pgfmathparse{3*\spacex};
\node [draw=black, shape=circle, fill=black] (4) at (\pgfmathresult cm, \spacey cm) {};
\node [scale=.8] at ([yshift=-.3 cm]4) {$4$};
\foreach \x/\c in {7/0,6/1,5/2}
{
\pgfmathparse{\spacex*\c};
\ifthenelse{\x=5}
{
\node [draw=black, shape=circle, fill=white] (\x) at (\pgfmathresult cm, 2*\spacey cm) {};
}
{
\node [draw=black, shape=circle, fill=black] (\x) at (\pgfmathresult cm, 2*\spacey cm) {};
};
\node [scale=.8] at ([yshift=.3 cm]\x) {$\x$};
}
\draw (2) -- (1) -- (6) -- (7) -- (2);
\draw (2) -- (3) -- (4) -- (5) -- (6);
\end{tikzpicture}
\caption{The graph $G$ satisfies that $\Sd_s(G,G^c)=\dim_s(G)=4>3=\beta(G)$.}
\label{figSdDim}
\end{figure}
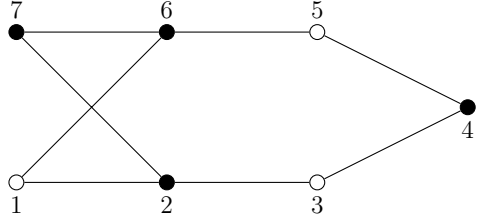

To illustrate the result above we take the graphs  shown in Figures 
 \ref{figSdBeta} and \ref{figSdDim}. In both cases  $D(G^c)=2$.
  Now, in the case of Figure \ref{figSdBeta},   the sets $\{1,5,6,7\}$ and $\{5,6,7,11\}$ are the only strong metric bases of $G$. At the same time, the set $\{1,5,6,7,11\}$ is the only $\beta(G)$-set which is a strong metric generator of $G$, and so it is the only $\beta_s(G)$-set. Therefore, $\Sd_s(G,G^c)=\beta_s(G)=\beta(G)=5>4=\dim_s(G)$.
In the case of 
Figure \ref{figSdDim},   $\Sd_s(G,G^c)=\beta_s(G)=\dim_s(G)=4>3=\beta(G)$, as $\{2,4,6,7\}$ is a strong metric basis of $G$ which is a vertex cover of $G$ and $\{2,4,6\}$ is a $\beta(G)$-set.

The hypercube $Q_r$, $r\ge 3$, of order $2^r$ is a $2$-antipodal graph and so $\dim_s(Q_r)=2^{r-1}$. Also, $Q_r$  is a bipartite graph and, for $r$ odd, any colour class form a strong metric basis  which is a vertex cover of minimum cardinality. Since $D(Q_r^c)=2$, we conclude that for any odd integer $r\ge 3$, 
\begin{equation}\label{FormulaHypercube}
\Sd_s(Q_r,Q_r^c)=\dim_s(Q_r)=\beta(Q_r)=2^{r-1}.
\end{equation} 
This is an example where $\Sd_s(G,G^c)=\dim_s(G)=\beta(G)$ and it is a particular case of the next result.

\begin{proposition}\label{Propo2-antipodalBipartiteOdd}
For any  bipartite $2$-antipodal graph $G$ of odd diameter  and order $n>2$,
$$\Sd_s(G,G^c)=\frac{n}{2}.$$
\end{proposition}

\begin{proof}
Let $G=(V_1\cup V_2,E)$. Since the subgraph of $G^c$   induced by $V_i$, $i\in \{1,2\}$, is complete and $G$ is not a complete bipartite graph, we conclude that $G^c$ is connected.
Furthermore, since $G$ is $2$-antipodal of odd diameter, each vertex $x\in V_1$ is adjacent to a vertex $x'\in V_2$ in $G^c$ and, as a result, $D(G^c)=2$.

On the other hand, $V_1$ is a vertex cover of $G$ and since $G$ is a $2$-antipodal graph and $D(G)$ is odd, for any $x\in V_1$ there exists exactly one vertex $x'\in V_2$ which is antipodal to $x$, which implies that $V_1$ is a strong metric basis of $G$. Therefore, by Corollary \ref{CorollaryEqualityto-dim(G)} we conclude the proof.
\end{proof}

An even-order cycle $C_{2k}$ has odd diameter if $k$ is odd, thus $\Sd_s(C_{2k},(C_{2k})^c)=k$ if $k$ is odd. Note that for $k$ even, $\Sd_s(C_{2k},(C_{2k})^c)=k+1$. 

If $G$ is a bipartite $2$-antipodal  graph, then the Cartesian product graph $G\Box K_2$ is bipartite and $2$-antipodal. Moreover, $D(G\Box K_2)=D(G)+1$. Therefore, Proposition \ref{Propo2-antipodalBipartiteOdd} immediately leads to the following result.

\begin{corollary}
For any  bipartite $2$-antipodal graph $G$ of even diameter  and order $n$,
$$\Sd_s(G\Box K_2,(G\Box K_2)^c)=n.$$
\end{corollary}

\begin{theorem}\label{GEqualGcSR}
Let $G$ be a connected graph. Then $G_{SR}=G^c$ if and only if  $D(G)=2$ and $G$ is a true twin-free graph. 
\end{theorem}

\begin{proof}
(Necessity) Assume that  $G_{SR}=G^c=(V,E)$, and let $u,v\in V$ be two mutually maximally distant vertices in $G$.

First consider that $u$ and $v$ are diametral vertices in $G$. Since $u$ and $v$ are mutually maximally distant in $G$ and $G_{SR}=G^c$, we obtain that $u$ and $v$ are adjacent in $G^c$ and, as a result, $D(G)=d_{G}(u,v)\ge 2$. 
Now, suppose that $d_{G}(u,v)>2$. Then  there exists $w\in N_{G}(v)-N_{G}(u)$ such that $d_{G}(u,w)=D(G)-1\ge 2$. Hence,  $w$ and $u$ are not mutually maximally distant in $G$ and $w\in N_G(u)$, which contradicts the fact that
$G_{SR}=G^c$. Therefore,  $D(G)=2$.

Now assume that $u$ and $v$ are true twins in $G$. We have that $u$ and $v$ are false twins in $G^c$ and, as a result,  they are not adjacent in  $G^c$ and they are mutually maximally distant in $G$, which contradicts  the fact that $G_{SR}=G^c$. Therefore, $G$ is a true twin-free graph.

(Sufficiency) If $G=(V,E)$ is a true twin-free graph and $D(G)=2$, then two vertices $u,v$ are mutually maximally distant in $G$ if and only if $d_{G}(u,v)=2$. Therefore, $G_{SR}=G^c$.
\end{proof}

Odd-order cycles are an example of the previous result, as $[(C_{2k+1})^c]_{SR}=C_{2k+1}$. Moreover, it is not difficult to show  that a simultaneous strong metric basis of $\{C_{2k+1}, (C_{2k+1})^c\}$ is the minimum union of a strong metric basis and a minimum vertex cover of $C_{2k+1}$, so $$\Sd_s(C_{2k+1}, (C_{2k+1})^c)=k+\left\lfloor\frac{k}{2}\right\rfloor+1.$$

\begin{corollary}\label{corollaryTwinsFree}
Let $G$ be a true twin-free graph such that $D(G)=2$. Then the following assertions hold. 
\begin{itemize}
\item $\Sd_s(G,G^c)=\dim_s(G)$ if and only if there exists a $\beta(G^c)$-set which is a strong metric generator for $G^c$.
\item $\Sd_s(G,G^c)=\dim_s(G)=\dim_s(G^c)$ if and only if there exists a  $\beta(G^c)$-set which is a strong metric basis of $G^c$.
\end{itemize}
\end{corollary}

The complement of the graph shown in Figure \ref{figSdBeta} has diameter two and $\{1,5,6,7,11\}$ is a $\beta(G)$-set which is a strong metric generator for $G$, so that $\Sd_s(G,G^c)=\dim_s(G)$.

Given a graph $G$, it is well-known that $D(G)\ge 4$ leads to $D(G^c)=2$. Hence,  $D(G )\ne 2$ and  $D(G^c)\ne 2$ if and only if $D(G )=D(G^c)= 3.$ In particular, for the case of trees we have that  $D(T )=3$  if and only if $D(T^c )=3.$

\begin{proposition}\label{treeDiam3andCompl}
Let $T$ be a tree of order $n$. If  $D(T)=3$, then $$\Sd_s(T,T^c)=n-2.$$
\end{proposition}

\begin{proof}
Notice that $T$ has  $l(T)=n-2$ leaves. Let $u$ and $v$ be the two interior vertices of $T$. We have that $D(T^c)=3$ and $d_{T^c}(u,v)=3$. Any simultaneous strong metric basis of $\{T,T^c\}$ must contain all leaves of $T$, except one, and one of $u$ and $v$, so $\Sd_s(T,T^c)\ge l(T)-1+1=n-2$. Moreover, by Corollary \ref{CorollarySD=n-1} we have  that $\Sd_s(T,T^c)\le n-2$ and so the equality holds.
\end{proof}

\begin{proposition}\label{treeDiamGe3andCompl}
Let $T$ be a tree of order $n$ such that $D(T) \ge 4$, let $l(T)$ be the number of leaves of $T$,  let $u$ be a leaf of $T$, and let $T'_u$ be the tree obtained from $T$ by removing all leaves, except $u$. Then, $$\beta(\mathring{T})+l(T)-1 \le \Sd_s(T,T^c) \le \beta(T'_u)+l(T)-1.$$
\end{proposition}

\begin{proof}
Note that $\dim_s(T)=l(T)-1$ and $\mathring{\beta}(T)=\beta(\mathring{T})$. Thus, by Corollary \ref{CorollaryInequalityMax}, $\Sd_s(T,T^c)\ge\max\{l(T)-1+\beta(\mathring{T}),\beta(T)\}$, and as a consequence, $\Sd_s(T,T^c)\ge \beta(\mathring{T})+l(T)-1$.

To prove the upper bound, let  $X$ be a   $\beta(T'_u)$-set and let $Y\subset V(T)$ be the set    composed by all leaves of $T$, except $u$. Notice that $X\cup Y$ is a strong resolving cover  of $T$ and $X\cap Y=\emptyset$. Now, since $D(T^c)=2$, by Corollary \ref{GEqualGcSRtwoDirection} we conclude  that  $\Sd_s(T,T^c)=\beta_s(T)\le |X|+|Y|=\beta(T'_u)+l(T)-1$.
\end{proof}

A particular case of the previous result 
is that of caterpillar trees $T$ such that
 $T'_u \cong P_{n-l(T)+1}$ for every leaf $u$ of $T$. In this case, we have that $\Sd_s(T,T^c)=l(T)+\left\lceil\frac{n-l(T)}{2}\right\rceil-1$. Moreover, if $D(T)=4$, then $\mathring{T}$ is a star graph. On the other hand, if $D(T)=5$, then $\mathring{T}$ is composed by exactly two interior vertices and $l(\mathring{T})=n-l(T)-2$ leaves. With these facts in mind, the following  two results are straightforward consequences of Proposition~\ref{treeDiamGe3andCompl}.

\begin{corollary}
\label{treeDiam4andCompl}
Let $T$ be a tree of order $n$ such that $D(T)=4$. If the central vertex of $\mathring{T}$ is a support vertex of $T$, then $$\Sd_s(T,T^c)=l(T).$$
Otherwise, $$\Sd_s(T,T^c)=l(T)+1.$$
\end{corollary}

\begin{corollary}
\label{treeDiam5andCompl}
Let $T$ be a tree of order $n$ such that $D(T)=5$. If an interior vertex of $\mathring{T}$ is a support vertex of $T$, then $$\Sd_s(T,T^c)=l(T)+1.$$
Otherwise, $$\Sd_s(T,T^c)=l(T)+2.$$
\end{corollary}

In general, the bounds shown in Proposition~\ref{treeDiamGe3andCompl} can be efficiently computed, as $\beta(T)$ can be computed in $O(n^{1.5})$ time for any tree $T$ \cite{Savage1982}.

\vspace{-0.4cm}
\section{Computability of the Simultaneous Strong Metric Dimension}\label{SectionComplexity}\label{SectionComputability}

It was shown in \cite{Oellermann2007} that the problem of finding the strong metric dimension of a graph, when stated as a decision problem, is $NP$-complete. This problem is formally stated as a decision problem as follows:

\medskip

\noindent{\bf Strong Metric Dimension (SDIM)}\\
INSTANCE: A graph $G=(V,E)$ and an integer $p$, $1 \le p \le |V(G)|-1$.\\
QUESTION: Is $\dim_s(G) \le p$?

\medskip

In an analogous manner, we define the decision problem associated to finding the simultaneous strong metric dimension of a graph family.

\medskip

\noindent{\bf Simultaneous Strong Metric Dimension (SSD)}\\
INSTANCE: A graph family ${\mathcal{G}}=\{G_1,G_2, \ldots, G_k\}$ defined on a common vertex set $V$ and an integer $p$, $1 \le p \le |V|-1$.\\
QUESTION: Is $\Sd_s({\mathcal{G}}) \le p$?

\medskip

It is straightforward to see that SSD is $NP$-complete in the general case, as determining whether a vertex set $S \subset V$, $|S| \le p$, is a simultaneous strong metric generator for a graph family ${\cal G}$ can be done in polynomial time, and for any graph $G=(V,E)$ and any integer $1 \le p \le |V(G)|-1$, the corresponding instance of SDIM can be transformed into an instance of SSD in polynomial time by making ${\cal G}=\{G\}$.

Here  we will discuss how the requirement of simultaneity makes computing the simultaneous strong metric dimension difficult, even for families composed by graphs whose individual strong metric dimension is easily computable. In particular, we will analyse the case of families composed by trees. As we have previously pointed out,  the  strong metric dimension
of any tree $T$ equals the number of leaves minus one, and  every set composed by all but one of its leaves is a strong metric basis \cite{Sebo2004}. Thus, for any tree $T$, a postorder traversal allows to compute $\dim_s(T)$ in polynomial time. However, the problem of finding the simultaneous strong metric dimension of a family of trees is $NP$-hard, as we will show. To this end, we will use a reduction of a subcase of the {Hitting Set Problem} (HSP), which was shown to be $NP$-complete by Karp \cite{Karp1972}. HSP is defined as follows:

\medskip

\noindent{\bf Hitting Set Problem (HSP)}\\
INSTANCE: A collection ${\cal C}=\{C_1,C_2,\ldots,C_k\}$ of non-empty subsets of a finite set $S$ and a positive integer $p\leq|S|$.\\
QUESTION: Is there a subset $S' \subseteq S$ with $|S'|\leq p$ such that $S'$ contains at least one element from each subset in ${\cal C}$?

\begin{theorem}
The SSD Problem is $NP$-complete for families of trees.
\end{theorem}

\begin{proof}
As we discussed previously, determining whether a vertex set $S \subset V$, $|S| \le p$, is a simultaneous strong metric generator for a graph family ${\cal G}$ can be done in polynomial time, so SSD is in $NP$.

It is known that HSP is $NP$-complete even if $|C_i|\leq 2$ for every $C_i \in {\cal C}$ \cite{Garey1979}. We will refer to this subcase of HSP as HSP2, and will show a polynomial time transformation of HSP2 into SSD. Let $S=\{v_1,v_2,\ldots,v_n\}$ be a finite set and let ${\cal C}=\{C_1,C_2,\ldots,C_k\}$, where every $C_i \in {\cal C}$ satisfies $1 \leq |C_i| \leq 2$ and $C_i \subseteq S$. Let $p$ be a positive integer such that $p\leq |S|$, and let $S'=\{w_1,w_2,\ldots,w_n\}$ such that $S \cap S'=\emptyset$. We construct the family ${\cal T}=\{T_1,T_2,\ldots,T_k\}$ composed by trees on the common vertex set $V=S \cup S' \cup \{u\}$, $u \notin S \cup S'$, as follows. For every $r \in \{1,\ldots,k\}$, if $C_r=\{v_{i_r}\}$, let $P_r$ be a path on the vertices of $(S-\{v_{i_r}\}) \cup (S'-\{w_{i_r}\})$, and let $T_r$ be the tree obtained from $P_r$ by joining by an edge the vertex $u$ to one end of $P_r$, and joining the other end of $P_r$ to the vertices $v_{i_r}$ and $w_{i_r}$. On the other hand, if $C_r=\{v_{i_r},v_{j_r}\}$, $P_r$ is a path on the vertices of $(S-\{v_{i_r},v_{j_r}\}) \cup S'$, and $T_r$ is the tree obtained from $P_r$ by joining by an edge the vertex $u$ to one end of $P_r$, and the other end of $P_r$ to the vertices $v_{i_r}$ and $v_{j_r}$. Figure~\ref{figTransformHSP2_SSMD} shows an example of this construction. We have that there exists a subset $S''$ of cardinality $|S''| \leq p$ that contains at least one element from each $C_i \in {\cal C}$ if and only if $\Sd_s({\cal T}) \leq p+1$. It is simple to verify that this transformation of HSP2 into SSD can be done in polynomial time.
\end{proof}

\begin{figure}[h]
\begin{center}

\begin{tikzpicture}
[inner sep=0.7mm, place/.style={circle,draw=black,
fill=white,thick},xx/.style={circle,draw=black!99,fill=black!99,thick},
transition/.style={rectangle,draw=black!50,fill=black!20,thick},line width=1pt,scale=0.5]
\coordinate (A1) at (0,1);
\coordinate (B1) at (1,1);
\coordinate (C1) at (2,1);
\coordinate (D1) at (3,1);
\coordinate (E1) at (4,1);
\coordinate (F1) at (5,1);
\coordinate (G1) at (6.5,1);
\coordinate (I1) at (7.5,0);
\coordinate (J1) at (7.5,2);

\coordinate (A2) at (9.5,1);
\coordinate (B2) at (10.5,1);
\coordinate (C2) at (11.5,1);
\coordinate (D2) at (12.5,1);
\coordinate (E2) at (13.5,1);
\coordinate (F2) at (14.5,1);
\coordinate (G2) at (16,1);
\coordinate (I2) at (17,0);
\coordinate (J2) at (17,2);

\coordinate (A3) at (19,1);
\coordinate (B3) at (20,1);
\coordinate (C3) at (21,1);
\coordinate (D3) at (22,1);
\coordinate (E3) at (23,1);
\coordinate (F3) at (24,1);
\coordinate (G3) at (25.5,1);
\coordinate (I3) at (26.5,0);
\coordinate (J3) at (26.5,2);

\draw[black] (A1) -- (B1) -- (C1) -- (D1) -- (E1) -- (F1) -- (G1);
\draw[black] (I1) -- (G1) -- (J1);

\draw[black] (A2) -- (B2) -- (C2) -- (D2) -- (E2) -- (F2) -- (G2);
\draw[black] (I2) -- (G2) -- (J2);

\draw[black] (A3) -- (B3) -- (C3) -- (D3) -- (E3) -- (F3) -- (G3);
\draw[black] (I3) -- (G3) -- (J3);

\node at (A1) [place]  {};
\coordinate [label=center:{$u$}] (u11) at (0,1.7);
\node at (B1) [place]  {};
\coordinate [label=center:{$v_3$}] (v31) at (1,0.3);
\node at (C1) [place]  {};
\coordinate [label=center:{$v_4$}] (v41) at (2,1.7);
\node at (D1) [place]  {};
\coordinate [label=center:{$w_1$}] (w11) at (3,0.3);
\node at (E1) [place]  {};
\coordinate [label=center:{$w_2$}] (w21) at (4,1.7);
\node at (F1) [place]  {};
\coordinate [label=center:{$w_3$}] (w31) at (5,0.3);
\node at (G1) [place]  {};
\coordinate [label=center:{$w_4$}] (w41) at (6,1.5);
\node at (I1) [place]  {};
\coordinate [label=center:{$v_2$}] (v21) at (7,-0.5);
\node at (J1) [place]  {};
\coordinate [label=center:{$v_1$}] (v11) at (7,2.5);

\node at (A2) [place]  {};
\coordinate [label=center:{$u$}] (u12) at (9.5,1.7);
\node at (B2) [place]  {};
\coordinate [label=center:{$v_1$}] (v12) at (10.5,0.3);
\node at (C2) [place]  {};
\coordinate [label=center:{$v_2$}] (v22) at (11.5,1.7);
\node at (D2) [place]  {};
\coordinate [label=center:{$v_4$}] (v42) at (12.5,0.3);
\node at (E2) [place]  {};
\coordinate [label=center:{$w_1$}] (w12) at (13.5,1.7);
\node at (F2) [place]  {};
\coordinate [label=center:{$w_2$}] (w22) at (14.5,0.3);
\node at (G2) [place]  {};
\coordinate [label=center:{$w_4$}] (w42) at (15.5,1.5);
\node at (I2) [place]  {};
\coordinate [label=center:{$w_3$}] (w32) at (16.5,-0.5);
\node at (J2) [place]  {};
\coordinate [label=center:{$v_3$}] (v32) at (16.5,2.5);

\node at (A3) [place]  {};
\coordinate [label=center:{$u$}] (u13) at (19,1.7);
\node at (B3) [place]  {};
\coordinate [label=center:{$v_1$}] (v13) at (20,0.3);
\node at (C3) [place]  {};
\coordinate [label=center:{$v_3$}] (v33) at (21,1.7);
\node at (D3) [place]  {};
\coordinate [label=center:{$w_1$}] (w13) at (22,0.3);
\node at (E3) [place]  {};
\coordinate [label=center:{$w_2$}] (w23) at (23,1.7);
\node at (F3) [place]  {};
\coordinate [label=center:{$w_3$}] (w33) at (24,0.3);
\node at (G3) [place]  {};
\coordinate [label=center:{$w_4$}] (w43) at (25,1.5);
\node at (I3) [place]  {};
\coordinate [label=center:{$v_4$}] (v43) at (26,-0.5);
\node at (J3) [place]  {};
\coordinate [label=center:{$v_2$}] (v23) at (26,2.5);

\coordinate [label=center:{$T_1$}] (T1) at (4,-1);
\coordinate [label=center:{$T_2$}] (T2) at (13.5,-1);
\coordinate [label=center:{$T_3$}] (T3) at (23,-1);

\end{tikzpicture}

\caption{The family ${\cal T}=\{T_1,T_2,T_3\}$ is constructed for transforming an instance of HSP2, where $S=\{v_1,v_2,v_3,v_4\}$ and $C =\{\{v_1,v_2\},\{v_3\},\{v_2,v_4\}\}$, into an instance of SSD.}
\label{figTransformHSP2_SSMD}
\end{center}
\end{figure}
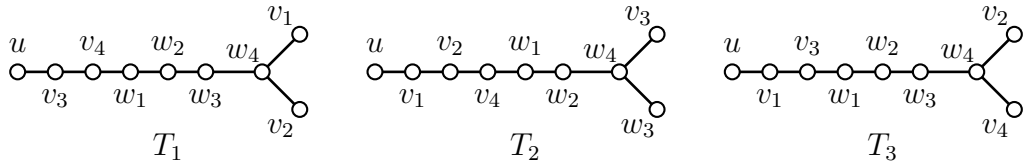

\end{document}